\documentclass[11pt, oneside]{amsart}       
\usepackage{geometry}                        
\usepackage[parfill]{parskip}            
\usepackage{graphicx}                
                            \usepackage{amsthm}    
\newtheorem{theorem}{Theorem}[section]
\newtheorem{corollary}{Corollary}[section]
\newtheorem{definition}{Definition}[section]
\newtheorem{lemma}[theorem]{Lemma}
\theoremstyle{remark}
\newtheorem*{remark}{Remark}
\usepackage{tcolorbox}

\title{Poisson Vertex Cohomology and Tate Lie Algebroids}
\author{Emile Bouaziz}
\begin{document}\begin{abstract} We study sheaves on holomorphic spaces of loops and apply this to the study of the complex, defined in \cite{BdSHK}, governing deformations of the \emph{Poisson vertex algebra} structure on the space of holomorphic loops into a Poisson variety. We describe this complex in terms of the (continuous) de Rham-Lie cohomology of an associated Lie algebroid object in locally linearly compact topological (alias \emph{Tate}) sheaves of modules on $\mathcal{L}^{+}M$. In particular this allows us to easily compute the cohomology of the above in the case where $\pi$ is symplectic - we obtain de Rham cohomology of $M$. \end{abstract}
\maketitle

\section{Introduction}

If $A_{\hbar}$ is an associative algebras over $k[[\hbar]]$ with the commutators vanishing to order $\hbar$, then the $\hbar\rightarrow0$ limit of the family inherits the structure of a Poisson algebra. The appropriate \emph{chiral} version of this statement says that a vertex algebra $V_{\hbar}$, whose underlying Lie-conformal algebra is commutative to order $\hbar$ (so that Laurent terms in the OPEs of $V_{\hbar}$ vanish to order $\hbar$) is naturally a \emph{Poisson vertex algebra}, that is to say a commutative vertex algebra equipped with a sort of Poisson bracket. In fact somewhat more is true, any vertex algebra is \emph{canonically filtered}, and the associated graded with respect to this filtration is a Poisson vertex algebra.  \\ \\ If one wishes to study vertex algebras (their representation theory, deformation theory etc) a useful first approximation is then to study the somewhat easier to understand Poisson vertex algebras (PVAs henceforth). Notice in particular that the deformation theory of a PVA should encode whether or not it can arise interestingly as the degeneration of a family of vertex algebras. The deformation theory of PVAs (and vertex algebras) has recently been investigated in cohomological and operadic terms in the work of Bakalov, de Sole, Heluani and Kac (\cite{BdSHK}), which we will use as our reference for the theory of PVAs more broadly. \\ \\ We will concern ourselves in ths note with the Poisson geometry of a smooth Poisson variety $(M,\pi)$. It is natural to assume that the Poisson structure on $M$ should correspond to a PV structure on some larger space associated to $M$ (that is to say a larger space whose structure sheaf is naturally a sheaf of PVAs). Perhaps the most natural guess is the space $\mathcal{L}^{+}M$ of holomorphic loops into $M$, ie the space with $R$-points $M(R[[z]])$. That this indeed inherits a natural PV structure has been shown by Arakawa \cite{A}. It is this space and its PV geometry with which we deal primarily. \\ \\ We begin by studying the theory of certain sheaves on the space of \emph{arcs} (aliases \emph{jets} or \emph{holomorphic loops}) $\mathcal{L}^{+}M$, for a scheme $M$. We emphasize throughout that the constructions can be characterized universally if we work in a category of suitably decorated spaces, namely $\delta$-spaces. We make use of these universal constructions to show that Lie algebroids on $M$ correspond naturally to certain topological Lie algebroids on $\mathcal{L}^{+}M$. In the case of a Poisson variety $(M,\pi)$ we can encode $\pi$ as a Lie algebroid structure on $\Omega^{1}_{M}$. We will then see that the associated topological Lie algebroid on $\mathcal{L}^{+}M$ governs the PV deformation theory, which allows us to explictly compute this deformation theory in the case that $\pi$ is non-degenerate. More precisely we will see that the following holds, \begin{theorem} Let $(M,\pi)$ be a smooth Poisson variety with non-degenerate $\pi$. Then there is an isomorphism: $$H^{*}_{PV}(\mathcal{L}^{+}M,\mathcal{O})\cong H^{*}_{dR}(M).$$\end{theorem}

\section{acknowledgements}  We have benefited from numerous discussions regarding vertex deformations with Ian Grojnowski, Vasilli Gorbounov and Owen Gwilliam. We wish especially to thank Reimundo Heluani for answering a relevant Mathoverflow question of the author, as well as for his very helpful remarks on a draft of this note.

\section{Holomorphic Loop Spaces and $\delta$-Schemes}\subsection{Definition of Loops} We give a very brief introduction to the basic theory of spaces of holomorphic loops, emphasizing the canonical global vector field. We make no claim to originiality here. Further, we make no mention of Poisson vertex algebras in this section. For more detailed references we refer the reader to \cite{DF}. We work throughout over an algebraically closed field of characteristic $0$, which we denote $k$. One is supposed to imagine that $\mathcal{L}^{+}X:=Maps(D_{z},X)$ where $D_{z}$ is a formal disc with the property that $D_{z}\times spec(A)=spec(A[[z]])$. \begin{remark}Note that some subtleties arise with this heuristic as it is not the case that $A\otimes k[[z]]=A[[z]]$ in general. Even working with the honest formal spectrum $spf k[[z]]$ one would have to prove that the definition given below agrees with $Maps(spf k[[z]],-)$, that is to say that for suitable $X$ we have $X(R[[z]])\cong\lim_{n}X(R[z]/z^{n+1})$. This is proven by Bhatt for qcqs schemes, it is a hard result and not needed for this paper. \end{remark}

\begin{definition} If $M$ is a scheme over $k$, we define the holomorphic loop space, denoted $\mathcal{L}^{+}M$ to be the presheaf on $Aff_{k}$ defined by $\mathcal{L}^{+}M(R)=M(R[[z]])$. The presheaf $\mathcal{L}^{+}_{n}M$ are defined to have $R$-points $M(R[z]/z^{n+1})$.\end{definition}

\begin{lemma} The presheaf $\mathcal{L}^{+}M$ is representable by a scheme, as are the presheaves $\mathcal{L}^{+}_{n}X$. \end{lemma}
\begin{proof} \emph{Sketch}. One checks the result for $M=\mathbb{A}^{1}$ and observes compatibility with the formation of limits to deduce it for arbitrary affine schemes. One then checks compatability with Zariski colimits to conclude. \end{proof}

We collect below some simple properties of the spaces $\mathcal{L}^{+}M$. \begin{remark}\begin{itemize}

\item Let $A$ be a $k$-algebra. We define the $k$-algebra $\mathcal{L}^{+}A$ to be generated by symbols $a_{i}$, for $a\in A$ and $i$ a non-negative integer, subject to the relations $(ab)_{n}=\sum a_{i}b_{n-i}$. Let $X=spec(A)$ be an affine $k$-scheme, then we have $\mathcal{L}^{+}X\cong spec(\mathcal{L}^{+}A)$.
\item The $\mathbb{G}_{m}(R)$ action on $R[[z]]$ endows the space $\mathcal{L}^{+}M$ with an action of $\mathbb{G}_{m}$. Locally this corresponds the grading on the ring of functions $\mathcal{L}^{+}A$ defined by stipulating that $a_{i}$ have degree $i$. This is usually referred to as the \emph{conformal} grading.
\item There is an evaluation morphism $$ev_{z}:(\mathcal{L}^{+}M)[[z]]\rightarrow M.$$ Here if $X$ is a space, locally $spec(A)$ we write $X[[z]]$ for the space locally of the form $spec(A[[z]])$. This evaluation morphism is gotten from the functor of points description as a universal map. In local coordinates it takes the form $a\mapsto \sum a_{i}z^{i}$.
\item There is a global vector field, $\delta_M$ on $\mathcal{L}^{+}M$, defined locally by $\delta_{A}(a_{i})=(i+1)a_{i+1}$. 
\end{itemize}\end{remark}
\subsection{$\delta$-Schemes}

In fact, it is often appropriate to consider the vector field $\delta_{M}$ as additional structure on $\mathcal{L}^{+}M$, and to insist that all our manipulations respect this structure. It will simplify exposition somewhat to introduce the category whose objects are pairs $(X,\delta)$ with $\delta$ a global vector field on $X$. Morphisms in this category are defined in the evident manner. We denote this category $Sch_{k}^{\delta}$ and refer to its objects as $\delta$-schemes. The opposite to the category of affine $\delta$-schemes will be called the category of $\delta$-algebras and denoted $Alg^{\delta}_{k}$. Let us observe that there is a forgetful functor $ \phi:Sch^{\delta}_{k}\rightarrow Sch_{k}$. We are now in a position to give an elegant characterization of the functor $X\mapsto \mathcal{L}^{+}X$ (equipped with its $\delta$-structure). We shall refer henceforth to $(\mathcal{L}^{+}M,\delta_{M})\in Sch^{\delta}_{k}$ as simply $\mathcal{L}^{+}_{\delta}M$, and define similarly $\mathcal{L}^{+}_{\delta}\in Alg^{\delta}_{k}$.

\begin{lemma} There is an adjunction $$Map_{Sch^{\delta}_{k}}(Y,\mathcal{L}^{+}_{\delta}X)\cong Map_{Sch_{k}}(\phi(Y),X).$$ \end{lemma}

\begin{proof} Note that we can work locally to check this. In this case it says that a morphism $f: A\rightarrow B$, to an algebra $B$ equipped with a derivation $\delta$, can be uniquely extended to a morphism $\mathcal{L}^{+}A$ which is compatible with $\delta_{A}$. We abbreviate $\delta^{(i)}:=\frac{\delta^{i}}{i!}$ throughout. The extension is obviously defined by demanding that $a_{i}$ map to $\delta^{(i)}(f(a))$. The Leibniz rule computing $\delta^{(n)}(xy)$ is identical to the relations imposed on the generators $a_{i}$ and thus we see that this is well defined, and we are done. \end{proof}

We will now consider the theory of quasi-coherent sheaves on $\delta$-schemes. 

\begin{definition} Let $(A,\delta)\in Alg^{\delta}_{k}$. We define the category $Mod(A,\delta)$ to consist of pairs $(M,\delta_{M}:M\rightarrow M)$ satisfyinging $\delta_{M}(am)=a\delta_{M}(m)+\delta(a)m$, ie $\delta_{M}$ is a first order differential operator with symbol $\delta\otimes id\in \Theta_{A}\otimes_{A}Map_{A}(M,M)$. Morphisms are defined in the obvious manner. We globalize this to define $QC(X,\delta)$ for $(X,\delta)\in Sch^{\delta}_{k}.$ \end{definition} Below we collect some simple properties of these categories.

\begin{remark}\begin{itemize} \item There is a forgetful functor $ QC(X,\delta)\rightarrow QC(X)$ sending $(\mathcal{F},\delta)$ to $F$. 
\item $QC(X,\delta)$ has a canonical element $(\mathcal{O}_{X},\delta)$, which we simplify to $\mathcal{O}_{X}$.
\item $QC(X,\delta)$ admits internal homs and tensor products.
\item Writing $X^{\delta}$ for the scheme of zeroes of $\delta$, there is an identification of elements $F\in QC(X,\delta)$ so that $\delta_{F}=0$ with $QC(X^{\delta}).$ The left adjoint to the inclusion of such sheaves is given as $F\mapsto F/\delta$.
\item The cotangent sheaf, $\Omega^{1}_{X}$ admits the structure of a $\delta$-sheaf, as does the tangent sheaf $\Theta_{X}$. The de Rham differential on the cotangent sheaf and the Lie bracket on the tangent sheaf are compatible with the $\delta$-structure.
\end{itemize}\end{remark}\subsection{Natural $\delta$-Sheaves on $\mathcal{L}^{+}X$} It is the goal of this subsection to introduce two functors $QC(X)\rightarrow QC(\mathcal{L}^{+}_{\delta}X)$. These functors will allow us to describe the cotangent and tangent sheaf on $\mathcal{L}^{+}X$ (with their $\delta$-structures.) Finally, we will show that these functors are dual in an appropriate sense (they both take rather large values but the duality is manageable if we work with Tate objects.) We will often work locally and simply remark here that the results all globalize verbatim. \begin{definition} We define a functor $\mathcal{L}^{+}: Mod_{A}\rightarrow Mod(\mathcal{L}^{+}_{\delta}A)$ as follows: $\mathcal{L}^{+}M$ is generated over $\mathcal{L}^{+}A$ by symbols $m_{i}$ for $m\in M$ and $i$ a non-negative integer, subject to the relations $(am)_{j}=\sum a_{i}m_{j-i}$. The $\delta$-structure is defined so that $m_{i}\mapsto (i+1)m_{i+1}$. \end{definition}

\begin{lemma} Write $\pi_{*}$ for the natural functor $Mod(\mathcal{L}^{+}_{\delta}A)\rightarrow Mod_{A}$. Then $\mathcal{L}^{+}$ is the left adjoint to the functor $\pi_{*}$.  \end{lemma}

\begin{proof} Note that, for $M$ a module for $A$, and $N$ a module for $\mathcal{L}^{+}_{\delta}A$  we must show simply that we have $$Map_{\mathcal{L}^{+}_{\delta}A}(\mathcal{L}^{+}M,N)\cong Map_{A}(M,\pi_{*}N).$$ We simply extend the morphism $M\rightarrow \pi_{*}N$ in the only way possible to all of $\mathcal{L}^{+}M$ and observe that this is really a map of modules on $\mathcal{L}^{+}_{\delta}A$. \end{proof}

\begin{corollary} There is an identification of modules on $\mathcal{L}^{+}_{\delta}A$, $$\mathcal{L}^{+}\Omega^{1}_{A}\cong \Omega^{1}_{\mathcal{L}^{+}A}.$$ \end{corollary}
\begin{proof} The adjunction description implies this immediately, noting that $A$-derivations into $\pi_{*}N$ uniquely extend to $\delta$-compatible $\mathcal{L}^{+}A$-derivations into $N$. Note that explicitly this correspondence is given by $da_{i}\mapsto (da)_{i}$.\end{proof}

There is another functor of interest. To introduce it we switch to geometric language and consider the correspondence of spaces, $X\longleftarrow\mathcal{L}^{+}X[[z]]\longrightarrow\mathcal{L}^{+}X$, given by the evaluation morphism $ev_{z}$ and the natural projection $\pi_{z}$. Note that we can upgrade this correspondence to one of $\delta$-schemes. To do this we give $X$ the trivial $\delta$-structure. The $\delta$-structure on $\mathcal{L}^{+}X[[z]]$ is defined as $\delta_{X}-\partial_{z}$. That $ev_{z}$ respects $\delta$-structures follows by noting that (locally) we have $\sum_{i}\delta_{A}(a_{i})z^{i}=\partial_{z}\sum_{i}a_{i}z^{i}$.

\begin{definition} We define the functor $\int$ by $$\int:=\pi_{z*}ev_{z}^{*}: QC(X)\longrightarrow QC(\mathcal{L}^{+}_{\delta}X).$$\end{definition}

\begin{remark}  This can be defined quite generally as a functor $\int_{F}:QC(X)\rightarrow QC(Maps(F,X))$. It is instructive to compute first order deformations of a map $f\in Maps(F,X)$. We see that they are given by $\Gamma(F,f^{*}\Theta_{X})$, that is to say by the fibre at the point $f$ of $\int\Theta_{X}$.This compuation globalises to an identification $\int_{F}\Theta_{X}\cong \Theta_{Maps(F,X)}$. Recalling our heuristic definition of $\mathcal{L}^{+}$ as $Maps(D_{z},-)$, we should have $\int\Theta_{X}\cong\Theta_{\mathcal{L}^{+}X}$. We will see this below.\end{remark}

\begin{lemma} We have $\int\Theta_{X}\cong\Theta_{\mathcal{L}^{+}X}$ as objects of $QC(\mathcal{L}^{+}_{\delta}X)$. \end{lemma}
\begin{proof} We work locally as usual. Considering $\mathcal{L}^{+}_{\delta}A[[z]]$ as an $A$-algebra via the evaluation map $ev_{z}$ it suffices to prove that there is an equivalence of $\delta$-modules on $\mathcal{L}^{+}_{\delta}A$, $$\pi_{z*}Der_{A}(A,\mathcal{L}^{+}_{\delta}A[[z]])\cong Der_{\mathcal{L}^{+}_{\delta}A}(\mathcal{L}^{+}_{\delta}A,\mathcal{L}^{+}_{\delta}A).$$ Indeed one sends an element $\eta$ of the left hand side to the derivation on $\mathcal{L}^{+}_{\delta}A$ sending $a_{i}$ to the $z^{i}$-coefficient of $\eta(a)$. Note that this is little more than the observation that by the functor of points definition of $\mathcal{L}^{+}$, first order algebra deformations of $id_{\mathcal{L}^{+}}$ are first order deformations of $ev_{z}$. \end{proof}

\begin{remark} We note that in the case of $F=\Omega^{1}_{X}$ we have proven that $\int F^{\vee}\cong (\mathcal{L}^{+}F)^{\vee}.$ We would like to generalise this to arbitrary locally free sheaves on $X$ (better perfect complexes). The complicating factor is the size of $\mathcal{L}^{+}X$. Indeed even for $X$ smooth one cannot expect a duality between its cotangent and tangent sheaves. This is remedied by noting that for $F$ locally free, the sheaves $\mathcal{L}^{+}F$ and $\int F$ are both naturally given the structure of Tate objects of $Pro(QC(\mathcal{L}^{+}_{\delta}X))$. We recall here that the category of Tate sheaves on $X$ is defined to the smallest extension and summand closed subcategory of $Pro(QC(X))$ containing $QC(X)$ and $Pro(Perf(X))$. Locally we can identify this category with $\mathcal{O}(X)$-modules endowed with a locally linearly compact topology. This category inherits a natural duality functor. For a proper reference we recommend \cite{Dr}.\end{remark}\begin{definition} \begin{itemize}\item For $F$ a sheaf on $F$ we define the object $\mathcal{L}^{+}F\in Pro(QC(\mathcal{L}^{+}_{\delta}))$ to be discrete, ie equal to its image under the natural embedding  $QC\rightarrow Pro(QC)$. Topologically this corresponds to a discretely topologized sheaf.\item We define the object $\int F\in Pro(QC(\mathcal{L}^{+}_{\delta}X))$ to be the pro-system $\{\int F \;mod \,z^{n}\}_{n}$. The notation should be clear (we consider the evaluation morphism as a ind-system of morphisms given by killing powers of $z$ and take the pro-system of functors on sheaves corresponding to the ind-system of correspondences).\end{itemize} \end{definition}

\begin{lemma} If $F$ is a locally free sheaf then the functor $\int$ maps $F$ to a Tate object of $Pro(QC(\mathcal{L}^{+}_{\delta}X))$, as does the functor $\mathcal{L}^{+}$.\end{lemma}
\begin{proof} Note that in the case of $\mathcal{L}^{+}$ (which we have discretely topologized) there is nothing to prove. We deal now with the case of $\int$. We may work locally, and thus assume that $F$ is the summand of a free $A$-module. $Tate(\mathcal{L}^{+}_{\delta}X)$ is defined to be closed under summands and thus we are reduced to the case of the module $A:=\mathcal{O}(X)$. In this case we produce the pro-object $\mathcal{L}^{+}_{\delta}A[[z]]/z^{n+1}$, which is a pro-diagram of projective modules and thus is Tate. \end{proof}

We come now to one of the main results of this subsection, namely that the functors $\mathcal{L}^{+}$ and $\int$ are interwtined by duality. Note that it is (at least to the author) slightly surprising that these two functors are so nicely related given their differing definitions - one arising as a certain free $\delta$-sheaf and the other from a functor defined in terms of a universal mapping space correspondence. 

\begin{lemma} For a locally free sheaf $F$ on $X$, there is an isomorphism in $Tate(\mathcal{L}^{+}_{\delta}X)$, $$(\mathcal{L}^{+}F)^{\vee}\cong\int F^{\vee}.$$ \end{lemma}
\begin{proof} We work locally, writing $A=\mathcal{O}(X)$ and $M$ for the module corresponding to $F$. We denote the pairing $M\otimes M^{\vee}\longrightarrow A$ by $(-,-)$. We define a pairing (also denoted $(-,-)$), $$(-,-): \int M^{\vee}\otimes\mathcal{L}^{+}M\rightarrow\mathcal{L}^{+}_{\delta}A,$$ by sending the element $(m^{\vee}\otimes z^{j})\otimes m_{i}$ to $\delta^{(i-j)}(m^{\vee},m)$, where negative powers of $\delta$ are defined to be zero.\\  \\Let us first note that this is well defined, indeed $(am^{\vee}\otimes z^{j})\otimes (bm)_{i}$ is sent to $$\delta^{(i-j)}(ab(m^{\vee},m))=\sum_{l,k}\delta^{(l)}(a)\delta^{(k)}(b)\delta^{(i-j-k-l)}(m^{\vee},m),$$ which is indeed the image of $\sum_{l,k}a_{l}m^{\vee}\otimes z^{j+l}\otimes b_{k}m_{i-k}$, so that the pairing respects the relations defining $\int$ and $\mathcal{L}^{+}$. Further, note that this pairing is continuous with respect to the topology on $\int M^{\vee}$, since it vanishes for sufficiently high $j$. Finally to see that this pairing is non-degenerate we first note that $\mathcal{L}^{+}M$ is naturally the colimit of finite type submodules $\mathcal{L}^{+}_{n}M$ spanned by $m_{i}$ for $i$ at most $n$. (We caution the reader that these are not sub- $\delta$-modules.) The pairing  $$(-,-): \int M^{\vee}\otimes\mathcal{L}^{+}_{n}M\rightarrow\mathcal{L}^{+}A$$ descends to a pairing  $$(-,-): (\int M^{\vee}\,mod\,z^{n+1})\otimes\mathcal{L}^{+}_{n}M\rightarrow\mathcal{L}^{+}A,$$ and it suffices to show that each of these is non-degenerate. \\ \\ Now observe that $\mathcal{L}^{+}_{n}M$ admits a natural increasing filtration of length $n$ whose associated graded in each degree $i=0,...,n$ is simply a copy of $M\otimes_{A}\mathcal{L}^{+}A$. Dually $(\int M^{\vee}\,mod\,z^{n+1})$ admits a length $n$ decreasing filtration whose associated graded in graded in each degree $i=0,...,n$ is  a copy of $M^{\vee}\otimes_{A}\mathcal{L}^{+}A$. The pairing is compatible with these filtrations and so descends to associated gradeds, in which case it is a sum of the natural pairing between $M^{\vee}\otimes_{A}\mathcal{L}^{+}A$ and  $M\otimes_{A}\mathcal{L}^{+}A$, whence we are done. 
\end{proof} 

\begin{remark} In the case of the trivial one dimensional module $A$, this produces the familiar topological duality between $k[z]$ and $k((z^{-1}))$.  \end{remark}

\subsection{Lie Algebroids and $\int$}

We now continue to study the functor $\int$, noting that the above implies that this is essentially equivalent to studying $\mathcal{L}^{+}$. In the subsequent chapter we would like to study the Poisson Vertex geometry on the loop space of a Poisson variety in terms of Lie algebroids on the loop space. We recall here (to fix notation) only that a Lie algebroid on $X$ is a quasi-coherent sheaf $L$, with a \emph{bracket} map, $[-,-] : L\otimes_{k} L\rightarrow L$, and an $\mathcal{O}$-linear map $\rho: L\rightarrow \Theta_{X}$, called the \emph{anchor}, so that the bracket makes $L$ into a sheaf of $k$ Lie algebras, and we have (locally) that $[l_{1},al_{2}]=a[l_{1},l_{2}]+\rho(l_{1})(a)l_{2}$, for a function $a$ and sections $l_{i}$  of $L$.

The goal of this subsection then is to establish the following (very easy) lemma. \begin{lemma} Let $L$ be a locally free Lie algebroid on $X$. Then the Tate sheaf $\int L$ on $\mathcal{L}^{+}X$ is naturally a Lie algebroid with anchor map $$\int\rho:\int L\longrightarrow \int\Theta_{X}\cong\Theta_{\mathcal{L}^{+}X}.$$\end{lemma}

\begin{remark} We should mention why this is not \emph{completely} formal. $\int$ is indeed lax-monoidal, that is we are given natural maps $\int V\otimes\int W\rightarrow\int V\otimes W$, however note that both of these tensor products are taken over $\mathcal{O}$, whereas the bracket map defining a Lie algebroid is not assumed to be $\mathcal{O}$-linear. The fact that $\int$ possesses some additional functoriality with respect to certain non-linear maps of modules is then the point.\end{remark} 

\begin{definition} A map $f: V\rightarrow V$ of (the underlying $k$-vector spaces of) an $A$-module $V$ is said to be a \emph{derivation} if there exists $\sigma(f)\in \Theta_{A}$ so that we have $f(av)=af(v)+\sigma(f)(a)v$ for all $a\in A, v\in V$. An $n$- \emph{polyderivation} of $V$ is a pair consisting of an alternating $k$- linear map, $f: \bigwedge^{n}_{k}V\longrightarrow V$, and an $A$-linear map $\sigma(f)$, called the \emph{symbol} of $f$, $\sigma(f):\bigwedge^{n-1}_{A}V\longrightarrow\Theta_{A}$, satisfying the \emph{symbol condition}, $$f(av_{1}\wedge...\wedge v_{n})=af(v_{1}\wedge...\wedge v_{n})+\sigma(f)(v_{2}\wedge...\wedge v_{n})(a)v_{1}.$$ \end{definition}

\begin{remark} The relevance of this definition to the theory of Lie algebroids should be fairly evident, the pair consisting of a Lie algebroid bracket map and its anchor map form a $2$-polyderivation. For a detailed account of how a certain operad of polyderivations governs Lie algebroid structures on a module see the paper \cite{CM}.\end{remark}

\begin{lemma} Let $(f,\sigma(f))$ be an $n$-polyderivation pair of an $A$-module $V$. Then there exists a natural $n$-polyderivation pair $(\widehat{f},\sigma(\widehat{f}))$ of $\int V$ which is compatible with the $\delta$-structures and so that $\sigma(\widehat{f})$ is the natural morphism coming from the lax-monoidal structure. $$\bigwedge^{n-1}\int V\longrightarrow \int \bigwedge^{n-1}V\longrightarrow\int\Theta_{A}\cong\Theta_{\mathcal{L}^{+}A}.$$\end{lemma}

\begin{proof} We prove this for $n=1$ and remark that it is similar in general. So we are given a vector field $\sigma\in\Theta_{A}$ and a map $f:V\rightarrow V$ so that we have $f(av)=af(v)+\sigma(a)v$ identically. Let us first note that $\sigma$ extends uniquely to a derivation of $\mathcal{L}^{+}A$ in a manner compatible with $\delta$-structures. We denote the resulting vector field on $\mathcal{L}^{+}A$ by $\widehat{\sigma}$. We note that it is characterized as the image of the canonical section $1\in\int A$ under the map $\int A\rightarrow \int\Theta_{A}$. As such it is this $\widehat{\sigma}$ that will be our symbol map. Once we have stipulated this, there is only one way to define $\widehat{f}$ so that it has symbol $\widehat{\sigma}$ and repsects the $\delta$-structure. One checks that this works, and thus the lemma is proven.\end{proof}

\begin{corollary} Let $L$ be a locally free Lie algebroid on $X$. Then the Tate sheaf $\int L$ on $\mathcal{L}^{+}X$ is naturally a Lie algebroid with anchor map $$\int\rho:\int L\longrightarrow \int\Theta_{X}\cong\Theta_{\mathcal{L}^{+}X}.$$ \end{corollary}
 \begin{proof} This is now obvious. \end{proof}

\begin{remark} Given a Lie algebroid, $L$, on a space $X=spec(A)$ one can form its complex of so called \emph{de Rham- Lie} cochains. This is simply a Lie cohomology type differential on the graded space $sym_{A}(L^{\vee}[-1])$.We denote it $C^{*}_{dR,Lie}(X,L)$. For example if $L$ is the tautological Lie algebroid $\Theta_{X}$ we obtain de Rham cohomology. When the module $L$ is given a suitable topology it is of course better to take continuous cochains. In our case, we see that given a locally free Lie algebroid $L$, we obtain a natural complex $C^{*,cont}_{dR,Lie}(\mathcal{L}^{+}X,\int L)$. Further we note that by above the underlying graded module of this complex is $sym_{\mathcal{L}^{+}A}(\mathcal{L}^{+}L^{\vee})$ and that there is a $\delta$-structure on the complex. We will explain in the next section how these Lie algebroid cohomology complexes describe deformations of the Poisson vertex structure of the loop space into a Poisson variety. \end{remark}

\begin{definition} Let $L$ be a locally free Lie algebroid on a smooth variety $X$. We define the \emph{big loop de Rham-Lie} cochain complex of $L$ by $$\mathcal{L}^{+}C^{*}_{dR,Lie}(X,L):=C^{*,cont}_{dR,Lie}(\mathcal{L}^{+}X,\int L).$$ We then define the \emph{loop de Rham-Lie} cochain complex as the $\delta$-reduced cochain complex of the above. Explicitly we set $$\mathcal{L}^{+}_{\delta}C^{*}_{dR,Lie}(X,L):=C^{*,cont}_{dR,Lie}(\mathcal{L}^{+}X,\int L)/\delta.$$ \end{definition} 

\section{Poisson Vertex Geometry} \subsection{Definitions and Examples} We give an extremely brief introduction to the theory of PVAs, largely to fix some notation. The reader who wishes to see a detailed exposition is referred to \cite{BdSHK}. \begin{definition} A commutative algebra $A$ with derivation $\delta$ is called a \emph{Poisson vertex algebra} if it is endowed with a \emph{lambda bracket} $$\{\,_{\lambda}\,\}:A\otimes A\longrightarrow A[\lambda],$$ subject to the following axioms; \begin{enumerate}\item $\delta$ is a derivation for the bracket: $\delta\{a_{\lambda}b\}=\{\delta a_{\lambda}b\}+\{a_{\lambda}\delta b\}$.\item Sesquilinearity: we have $\{\delta a_{\lambda}b\}=-\lambda\{a_{\lambda}b\}.$\item Skew-symmetry: $\{a_{\lambda}b\}=-\{b_{-\delta-\lambda}a\}$.\item Jacobi: $\{\{a_{\lambda}b\}_{\lambda+\mu}c\}=\{a_{\lambda}\{b_{\mu}c\}\}-\{b_{\mu}\{a_{\lambda}c\}\}$.\item Leibniz: $\{a_{\lambda}-\}$ is a derivation of the commutative product on $A$.\end{enumerate} There is an evident category of such algebras which we denote $PVAlg_{k}$.\end{definition}

\begin{remark} This may look somewhat intimidating. It may help to first get a grasp on the properties of the bracket $\{_{\lambda}\}$ on the underlying vector space of $A$ with its $\delta$-structure. The axioms in this case make $\{_{\lambda}\}$ into a \emph{Lie conformal} algebra. Such objects have a nice interpretation as Lie objects in a suitable pseudo-$\otimes$ category of $k[\delta]$-modules, as is pointed out in the book of Beilinson-Drinfeld, \cite{BD}. Further, a Lie conformal algebra admits a universal enveloping PVA, and so these can be used to generate examples of PVAs. Another source of examples comes from the associated graded of a vertex algebra with respect to the so-called \emph{Li}-filtration, see for reference \cite{H}. We will not define vertex algebras in this note as we do not think it is needed and the definition is somewhat technical.\end{remark}

\begin{definition} A Poisson vertex scheme (PV scheme) is a scheme $X$ whose structure sheaf is endowed with the structure of a sheaf of PVAs. There is an evident category of such spaces which we denote $PVSch_{k}$. \end{definition} \begin{remark} Note that the structure of a PV scheme on $X$ automatically endows $X$ with the structure of a $\delta$-scheme.\end{remark}

\begin{lemma} Let $X$ be a PV scheme and let $X^{\delta}$ be the subscheme of zeroes of $\delta$. Then $X^{\delta}$ is naturally a Poisson scheme with bracket locally defined by $\{_{\lambda}\}_{\lambda=0}$. \end{lemma}
\begin{proof} That this is well defined and gives a Poisson structure can be easily checked from the axioms. \end{proof}

The following lemma is due to Arakawa, \cite{A}.\begin{lemma} Let $(X,\pi)$ be a Poisson scheme, then the scheme $\mathcal{L}^{+}X$ is naturally endowed with the structure of a PV scheme whose underlying $\delta$-scheme is simply $\mathcal{L}^{+}_{\delta}X,$, and so that the induced Poisson structure on $X=(\mathcal{L}^{+}_{\delta}X)^{\delta}$ is just $\pi$. \end{lemma} \begin{proof} We work locally, where this is an easy computation on account of the characterization of $\mathcal{L}^{+}_{\delta}A$ as a universal $\delta$ algebra. Let $\{,\}$ denote the Poisson bracket on $\mathcal{O}_{X}:=A$. Indeed we must have $$\{(a_{i})_{\lambda}(b_{j})\}=\delta^{(i)}({-\lambda}{-\delta})^{(j)}\{a_{\lambda}b\}$$ from the axioms for a PVA. It follows that defining $\{a_{\lambda}b\}:=\{a,b\}$ for $a,b\in A$ we there is a unique way to extend $\{_{\lambda}\}$ so that the Leibniz rule holds. One then simply checks that various axioms all hold and that the induced Poisson structure on $X$ is given by $\pi$, which is easy, and we are done. \end{proof}

\subsection{PVA Cohomology Complex} In order to define the PVA complex we first define the complex controlling deformations of a Lie conformal algebra. We denote the category of Lie conformal algebras $Lie^{*}_{k}$, following \cite{BD}. \begin{remark} We note that as $L\in Lie^{*}_{k}$ is a Lie algebra object in a suitable pseudo-$\otimes$ category there is actually a formal method to produce Chevalley-Eilenberg cochains on it from this. Indeed, in a symmetric monoidal category one can define multilinear maps from a finite collection $\{V_{i}\}_{i\in I}$ to an object $W$ as $$Mult(\{V_{i}\}_{i\in I},W):=Hom(\bigotimes_{i\in I}V_{i},W).$$ Forgetting the tensor product and axiomatizing the properties of these multinear maps one obtains the notion of a pseudo-tensor category. It is then not hard to define a Lie algebra object in such, as well as the Chevalley-Eilenberg cochain complex associated to it. See \cite{BD} for details.\end{remark}

\begin{definition} Let $L\in Lie^{*}_{k}$ be a Lie conformal algebra with Lie bracket $[\,_{\lambda}\,]$. The Lie conformal complex associated to $L$, denoted $C^{*}_{Lie^{*}}(L,L)$, is defined to have $n$-cochains maps $$Y:L^{\otimes n}\longrightarrow L[\lambda_{1},...,\lambda_{n}]/(\delta+\lambda_{1}+...+\lambda_{n}).$$ These are required to satisfy the sesquilinearity condition $$Y_{\lambda_{1},...,\lambda_{n}}(a_{1}\otimes...\otimes\delta a_{i}\otimes...\otimes a_{n})={-\lambda_{i}}Y_{\lambda_{1},...,\lambda_{n}}(a_{1}\otimes...\otimes a_{n}),$$ and the anti-symmetry condition $$Y_{\lambda_{1},...,\lambda_{i},\lambda_{i+1},...,\lambda_{n}}(a_{1}\otimes...\otimes a_{i}\otimes a_{i+1}\otimes...\otimes a_{n})=-Y_{\lambda_{1},...,\lambda_{i+1},\lambda_{i},...,\lambda_{n}}(a_{1}\otimes...\otimes a_{i+1}\otimes a_{i}\otimes...\otimes a_{n}).$$ The differential is defined by $$(dY)_{\lambda_{1},...,\lambda_{n}}(a_{1}\otimes...\otimes a_{n})= \sum_{i}(-1)^{i}[a_{i}\,_{\lambda_{i}}Y_{\lambda_{1},...,\widehat{\lambda_{i}},...,\lambda_{n}}(a_{1}\otimes...\otimes\widehat{a_{i}}\otimes...\otimes a_{n})]$$  $$+ \sum_{i,j}(-1)^{i+j}Y_{\lambda_{i}+\lambda_{j},\lambda_{0},...\widehat{\lambda_{i}},...,\widehat{\lambda_{j}},...,\lambda_{n}}([a_{i}\,_{\lambda_{i}}a_{j}]\otimes a_{0}\otimes...\otimes\widehat{a_{i}}\otimes...\otimes\widehat{a_{j}}\otimes...\otimes a_{n}).$$\end{definition}

 That this indeed squares to zero is proven in \cite{BKV}, the cohomology spaces will henceforth be denoted $H^{*}_{LC}(L,L)$. We remark also that there is a generalization of this definition where one takes cohomology in a representation of a Lie conformal algebra. To justify the name given to this complex one should prove that the low degree cohomology spaces have the interpretations one would expect, we refer to \cite{BKV} and \cite{DSK} for proofs that they really do have such interpretations. In particular they prove the following theorem; \begin{theorem} (\cite{BKV},\cite{DSK}). $H^{2}_{LC}(L)$ is the space of first order infinitesimal deformations of $L\in Lie^{*}_{k}$, which are trivial as deformations of the underlying $k[\delta]$-module. \end{theorem}
\begin{remark} It is now fairly clear how one should define the PVA complex. Indeed consider the case of classical Poisson algebras. One can construct the Poisson cohomology complex by taking the sub-complex of the Lie cohomology complex of the underlying Lie algebra and then restrict to the cochains which are polyderivations (of course one should check that the differential preserves this subcomplex.)\end{remark}

\begin{lemma} Let $A$ be a PVA. Then the subspace, denoted $C^{*}_{PV}(A,A)$ of $C^{*}_{LC}(A,A)$ formed by polydifferential operators is preserved by the differential. \end{lemma}
\begin{proof} See \cite{BdSHK}. \end{proof}

\begin{remark}\begin{itemize} \item Denoting the cohomology of this complex $H_{PV}^{*}(A,A)$, we have as expected an identification of $H^{2}_{PV}(A,A)$ as the space of first order infinitesimal deformations of $A$ which are trivial as deformations of the underlying $\delta$-algebra.\item For a PV scheme $X$ we denote the hypercohomology of the complex of sheaves $C^{*}_{PV}(\mathcal{O},\mathcal{O})$ by $H^{*}_{PV}(X,\mathcal{O})$\end{itemize}\end{remark}

\begin{theorem} Let $(M,\pi)$ be a smooth Poisson variety with non-degenerate $\pi$. Then there is an isomorphism: $$H^{*}_{PV}(\mathcal{L}^{+}M,\mathcal{O})\cong H^{*}_{dR}(M).$$  \end{theorem}
\section{Poisson Vertex Cohomology via Tate Lie Algebroid Cohomology}

We will now show how the formalism we developed for sheaves on $\mathcal{L}^{+}_{\delta}M$ can be used to describe the PVA complex on $\mathcal{L}^{+}M$ for a Poisson variety $(M,\pi)$. First, let us note that a Poisson structure on $M$ can be encoded as a Lie algebroid structure on the cotangent sheaf $\Omega^{1}_{M}$. The anchor map is $\pi:\Omega^{1}_{M}\rightarrow\Theta_{M}$ and bracket is defined locally by $\{da,db\}:=d\{a,b\}$. \\ \\By the results above we know then that the sheaf $\int\Omega^{1}_{M}\in Tate(\mathcal{L}^{+}_{\delta}M)$ is endowed with the structure of a topological Lie algebroid on $\mathcal{L}^{+}_{\delta}M$ with anchor $\int\pi:\int\Omega^{1}_{M}\rightarrow\int\Theta_{M}\cong\Theta_{\mathcal{L}^{+}_{\delta}M}$. Recall also that we have defined the loop de Rham Lie cohomology with coefficients in $L$ by $$\mathcal{L}^{+}_{\delta}C^{*}_{dR,Lie}(X,L):=C^{*,cont}_{dR,Lie}(\mathcal{L}^{+}X,\int L).$$  The following is now the key observation; \begin{lemma} There is a natural isomorphism of complexes of sheaves on $\mathcal{L}^{+}M$, $$\mathcal{L}_{\delta}^{+}C^{*}_{dR,Lie}(M,\Omega^{1}_{M})\cong C^{*}_{PV}(\mathcal{L}^{+}M,\mathcal{O}).$$\end{lemma} \begin{proof} Recall first that the graded sheaf of $\mathcal{L}^{+}M$-modules underlying $\mathcal{L}_{\delta}^{+}C^{*}_{dR,Lie}(M,\Omega^{1}_{M})$ is $sym_{\mathcal{O}_{\mathcal{L}^{+}M}}(\mathcal{L}^{+}\Theta_{M})$ using the duality between the functors $\int$ and $\mathcal{L}^{+}$. \\ \\ We now work locally and set $M=spec(A)$. Let us examine $n$-cochains in the PVA complex of $\mathcal{L}^{+}A:=B$. Such is given by a polyderivation $$Y_{\lambda_{1},...,\lambda_{n}}:B^{\otimes n}\longrightarrow B\,[\lambda_{1},...,\lambda_{n}]/(\delta+\sum\lambda_{i})$$ subject to the sesquilinearity and anti-symmetry conditions. The first thing to note is that freeness of the $\delta$-algebra $B$, combined with the sesquilinearity axiom, implies that $Y$ is determined entirely to its restriction to $A^{\otimes n}$. As such $Y$ can be encoded as an element of $$\Theta_{A}^{\otimes n}\otimes_{A}B[\lambda_{1},...,\lambda_{n}]/(\delta+\sum\lambda_{i}).$$ Further, the anti-symmetry condition translates into $Y$ giving an element of $\bigwedge^{n}(\Theta_{A}\otimes_{A}B[\lambda])/\delta$. Now note that there is a natural bijection of $A$-modules  $$\iota:\mathcal{L}^{+}\Theta_{A}\longrightarrow\Theta_{A}\otimes_{A}B[\lambda],$$ defined by $\iota(\eta_{i})=\eta\otimes\lambda^{i}$ for $\eta\in\Theta_{A}$. This extends to an isomorphism between the underlying graded $A$ modules of $\mathcal{L}_{\delta}^{+}C^{*}_{dR,Lie}(A,\Omega^{1}_{A})$ and $ C^{*}_{PV}(\mathcal{L}^{+}A,\mathcal{L}^{+}A)$. Further, this isomorphism intertwines the $\delta$-structures on both sides.\\ \\ We must show that this isomorphism intertwines the two differentials. Let us now see this at the $0$-th level. On the Poisson vertex side we have the map $b\mapsto\{b_{\lambda}-\}$ in degree $0$. On the loop de Rham Lie side we have the $0$th differential of the Lie algebroid $\int\Omega^{1}_{A}$. By definition this map is $$b\mapsto (\omega\mapsto (\int\pi)(\omega)(b)),$$ recalling that $\omega\in\int\Omega^{1}_{A}$ is defined to act on functions as $(\int\pi)(\omega)$. Unravelling this, one sees that under the identification $(\int\Omega^{1}_{A})^{\vee}=\mathcal{L}^{+}\Theta_{A}$, the differential takes the element $a_{i}$ to the element $(H_{a})_{i}\in\mathcal{L}^{+}\Theta_{A}$, where $H_{a}$ is the Hamiltonian associated to $a$. By definition the element $(H_{a})_{i}$ is sent to $H_{a}\otimes\lambda^{i}$ under $\iota$ and so we must simply observe that the definition of the PVA structure on $\mathcal{L}^{+}A$ implies that $\{a_{i}\,_{\lambda}-\}$ acts as $\lambda^{i}H_{a}$ on $A\subset\mathcal{L}^{+}A$. We now quotient by the action of $\delta$ and see that in the $0$-th degree the two differentials are intertwined. It is a similarly straightforward, if somewhat tedious, task to confirm the same for the higher differentials. \end{proof}

\begin{corollary}  Let $(M,\pi)$ be a smooth Poisson variety with non-degenerate $\pi$. Then there is an isomorphism: $H^{*}_{PV}(\mathcal{L}^{+}M,\mathcal{O})\cong H^{*}_{dR}(M).$\end{corollary}
\begin{proof} Note that in this case the map $\pi$ provides an isomorphism of Lie algebroids, $\Omega^{1}_{M}\longrightarrow\Theta_{M}$. Such induces an isomorphism of topological Lie algebroids $$\int\Omega^{1}_{M}\longrightarrow\int\Theta_{M}.$$Further we recall that we have an isomorphism $\Theta_{\mathcal{L}^{+}M}\cong\int\Theta_{M}$. As such we are reduced to computing the loop de Rham Lie cohomology of the tangent Lie algebroid, ie to computing the hypercohomology of  $\mathcal{L}_{\delta}^{+}C^{*}_{dR,Lie}(M,\Theta_{M})$. First we note that this is simply the hypercohomology of the $\delta$-reduced de Rham complex on $\mathcal{L}^{+}M$, ie the hypercohomology of $\Omega^{*}(\mathcal{L}^{+}M)/Lie_{\delta}$ with the de Rham differential. We claim that the natural inclusion  $$\Omega^{*}(M)\longrightarrow \Omega^{*}(\mathcal{L}^{+}M)/Lie_{\delta}$$ is a quasi-isomorphism. To do this let us note that the $\mathbb{G}_{m}$-action produces an Euler vector field $\eta$ on $\mathcal{L}^{+}M$. This acts diagonalizably on the de Rham complex via Lie derivative, $Lie_{\eta}$. It satisfies $[Lie_{\eta},Lie_{\delta}]=Lie_{\delta}$ and so preserves the image of $\delta$, and thus descends to a diagonalizable endomorphism of the complex $\Omega^{*}(\mathcal{L}^{+}M)/Lie_{\delta}$. If $\iota_{\eta}$ denotes contraction against the vector field $\eta$, then the Cartan formula $[d_{dR},\iota_{\eta}]=Lie_{\eta}$ implies immediately that the inclusion of the conformal weight $0$ subspace is a quasi-isomorphism. Now we note that as $Lie_{\delta}$ has image entirely of conformal weight strictly positive, the weight $0$ subspace is indeed $\Omega^{*}(M)$, whence the claim is proven. \end{proof}

\end{document}